\documentclass[11pt]{article}
\usepackage[utf8]{inputenc}
\usepackage{graphicx}
\usepackage{fullpage}
\usepackage{enumitem}
\usepackage{amsmath,amsthm,amssymb}
\usepackage[hidelinks]{hyperref}
\usepackage{multicol}
\usepackage{xcolor}
\usepackage{multirow}
\usepackage{float}
\usepackage{cleveref}
\usepackage{thm-restate}
\usepackage[backend = bibtex, style=alphabetic, url=false, isbn=false]{biblatex}
\DeclareFieldFormat{titlecase}{\MakeSentenceCase*{#1}}
\theoremstyle{definition}

\theoremstyle{remark}

\newtheorem*{fact*}{Fact}
\numberwithin{equation}{section}
\newtheoremstyle{probstyle}  %
{3ex}                        %
{0pt}                        %
{\normalfont}                %
{-3ex}                       %
{\normalfont\bfseries}       %
{.}                          %
{.5em}                       %
{}                           %
\theoremstyle{probstyle}

\newtheorem{lemma}{Lemma}

\usepackage{contour}
\usepackage[normalem]{ulem}

\contourlength{0.8pt}

\usepackage{listings}
\DeclareUrlCommand{\bfurl}{}
\DeclareUrlCommand{\urlstyle}{}

\begin{document}
\title{Small Ramsey numbers for books, wheels, and generalizations}
\author{
    Bernard Lidick\'y\footnote{Department of Mathematics, Iowa State University.  E-mail:  \href{mailto:lidicky@iastate.edu}{lidicky@iastate.edu}. Research supported by NSF DSM-2152490 and Scott Hanna Professorship.} 
    \and Gwen McKinley\thanks{Center for Communications Research, La Jolla. E-mail: \href{mailto:gmckinle@ccr-lajolla.org}{gmckinle@ccr-lajolla.org}}
    \and Florian Pfender\thanks{Department of Mathematical and Statistical Sciences, University of Colorado Denver. E-mail: \href{mailto:Florian.Pfender@ucdenver.edu}{Florian.Pfender@ucdenver.edu}. Research is partially supported by NSF grant DMS-2152498.}
    \and Steven Van Overberghe\thanks{
Department of Applied Mathematics, Computer Science and Statistics, Ghent University. \href{mailto:steven.vanoverberghe@ugent.be}{steven.vanoverberghe@ugent.be}
    }
}
\maketitle

\begin{abstract}
    In this work, we give several new upper and lower bounds on Ramsey numbers for \textit{books} and \textit{wheels}, including a tight upper bound establishing $R(W_5, W_7) = 15$, matching upper and lower bounds giving $R(W_5, W_9) = 18$, $R(B_2, B_8) = 21$, and $R(B_3, B_7) = 20$, and a number of additional tight lower bounds for books. We use a range of different methods: flag algebras, local search, bottom-up generation, and enumeration of polycirculant graphs.
    
    We also explore generalized Ramsey numbers using similar methods. 
    Let $GR(r,K_s,t)$ denote the minimum number of vertices $n$ such that any $r$-edge-coloring of $K_n$ has a copy of $K_s$ with at most $t$ colors. We establish $GR(3,K_4,2) = 10, GR(4,K_4,3) = 10$, and some additional bounds.
\end{abstract}

\section{Introduction}
The most classical problem in Ramsey theory deals with cliques in 2 colors: namely, it asks us to find the \textit{Ramsey number} $R(K_s, K_t)$, the smallest number $n$ guaranteeing that any red-blue edge-coloring of $K_n$ must contain a red copy of $K_s$ or a blue copy of $K_t$.
Equivalently, one may ask for the largest $n$ such that there exists a red-blue edge-coloring of $K_{n-1}$ avoiding the ``forbidden" subgraphs $K_s$ in red and $K_r$ in blue. There are many generalizations of this problem, including to larger numbers of colors, and %
forbidden subgraphs other than cliques. A dynamic survey by Radziszowski \cite{rad_survey} tracks the state of knowledge in this area. \\

Here, we will focus on Ramsey numbers where the forbidden
subgraphs are \textit{wheels} or \textit{books}. The \textbf{wheel} $W_k$ consists of a cycle on $k-1$ vertices (``spokes''), together with a central vertex adjacent to all vertices of the cycle. The \textbf{book} $B_k$ consists of $k+2$ vertices, an edge $uv$ (the ``spine''), and $k$ vertices (``pages") each adjacent only to $u$ and $v$. See Figure~\ref{fig-WB} for an illustration.\\

\begin{figure}
\begin{center}
    \begin{minipage}{.25\textwidth}
    \begin{center}
        $W_6$\\\bigskip

       \includegraphics{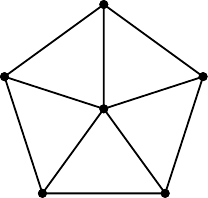}\mbox{}\\\bigskip
    \end{center}
   \end{minipage}\hspace{.07\textwidth}
   \begin{minipage}{.32\textwidth}
    \begin{center}
        \mbox{} $B_4$\\\bigskip

       \includegraphics{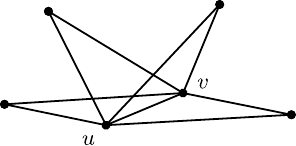}
    \end{center}
    \end{minipage}
\end{center}
\caption{Wheel $W_6$ and book $B_4$.}\label{fig-WB}
\end{figure}

There are a number of results on Ramsey numbers for books and wheels, summarized in \cite{rad_survey}. In particular, the first and third authors proved upper bounds for a variety of small Ramsey numbers, including for books and wheels, using the \textit{flag algebra method} \cite{LidP}; we establish most of the upper bounds here using the same method. The fourth author established a variety of lower bounds for graphs including wheels by explicit enumeration of \textit{circulant} and \textit{polycirculant} (or ``block-circulant") graphs~\cite{GOEDGEBEUR2022212}; here, we use the same technique to establish a number of new lower bounds for books. Other recent works have used reinforcement learning~\cite{RL_bounds} and SAT solvers~\cite{wesley} to give Ramsey lower bound constructions for books and wheels. 
\\

We also explore a generalization of Ramsey numbers
introduced by Erd\H{o}s and Shelah~\cite{MR0389669,MR0681869}.
In classical Ramsey problems, one is looking for a monochromatic copy of a graph $H$ in an $r$-edge-coloring of a graph $K_n$. 
The generalization asks for a copy of $H$ with at most $t$ colors for a fixed $t$.
We investigate the case where $H$ is a clique. 
Let $GR(r,K_s,t)$ denote the minimum number of vertices $n$ such that any $r$-edge-coloring of $K_n$ has a copy of $K_s$ with at most $t$ colors.
Observe
\[
GR(r,K_s,1) = R(\underbrace{K_s,\ldots,K_s}_{r}). 
\]
These generalized Ramsey numbers have been studied asymptotically; for some of the work see \cite{bal2024generalizedramseynumberscycles,gomezleos2024newboundsgeneralizedramsey,MR1645678,MR1656546,MR3872308}. The small values of $r,s,t$ for $GR(r,K_s,t)$ remained unexplored.

\section{Results}

We prove the following bounds; asterisks indicate bounds that are tight.\\

\textit{Note:} we have recently become aware of independent work establishing the same (tight) lower bound on $R(W_5, W_9)$ using an explicit construction~\cite{yanbo}. In~\cite{wesley2024}, some of the results below on Ramsey numbers for books are also proven, and there is also a generalization of the second case of Lemma~\ref{lem:bookslemmma} to an infinite family.

\begin{table}[H]
\begin{center}
    \begin{tabular}{l|ll|ll}
             & Lower (new) & Upper (new) & Lower (old)\ \ \  & Upper (old) \\ \hline
    $R(W_5,W_7)$ &             & \textit{15*} & 15 \cite{wesley}              & 16 \cite{LidP}         \\
    $R(W_5,W_9)$ & \textit{18*}               & \textit{18*} &    17 \cite{rad_survey}           &        \\
    \hline
    $R(B_2,B_8)$ & \textit{21}*               & \textit{21}*  & 19 \cite{RS78}                  &  22 \cite{b2_uppers}\\
    $R(B_2,B_9)$ & 22               & & 21 \cite{RS78}                  &  24 \cite{b2_uppers}\\
    $R(B_2,B_{10})$ & 25               & & 23 \cite{RS78}                  &  26 \cite{b2_uppers}\\
    $R(B_3,B_7)$ & $\textit{20}^\star$              & $\textit{20}^\star$ &            &          \\
    $R(B_4,B_7)$ & 22               & 23 &            &          \\
    $R(B_5,B_6)$ & \textit{23*}               & &                  & 23 \cite{RS78}           \\
    $R(B_6,B_7)$ & \textit{27*}               & &                  & 27 \cite{RS78}    \\
    $R(B_6,B_8)$ & \textit{29*}               & &             &  29 \cite{RS78}   \\
    $R(B_7,B_8)$ & \textit{31*}               & &               & 31 \cite{RS78}     \\
    $R(B_8,B_8)$ & \textit{33*}               & &          & 33 \cite{RS78}   \\
    \hline    
    $GR(3,K_4,2)$ & \textit{10*} & \textit{10*} & & \\
    $GR(3,K_5,2)$ & 20 & 23 & & \\
    $GR(3,K_6,2)$ & 32 & 54 & & \\
    $GR(4,K_4,2)$ & 15 & 17 & & \\
    $GR(4,K_4,3)$ & \textit{10*} & \textit{10*} & &     
    \end{tabular}
\end{center}
\caption{New upper and lower bounds on Ramsey numbers.}\label{table1}
\end{table}

We also have the following lemma covering the diagonal and almost-diagonal cases for some additional sizes of books. Note that some of these bounds are also included in the table above for readers' convenience.
\begin{lemma}\label{lem:bookslemmma}
    For all \(4 \leq n \leq 21 \) we have
    \begin{itemize}
        \item \(4n-3\leq R(B_{n-2},B_n) \)
        \item \(R(B_{n-1},B_{n}) = 4n-1 \)
        \item \(4n+1 \leq R(B_{n},B_{n}) \leq 4n+2 \)
    \end{itemize}
    Note that it was also already known that \(R(B_{n-2},B_{n})\leq 4n-3 \) if $n\equiv 2 \pmod 3$, and that \(R(B_{n},B_{n}) \leq 4n-1 \) if $4n+1$ is not the sum of two squares.
\end{lemma}

\section{Method}

Each of the lower bounds in the table above is established by a construction provided in \Cref{appendix}. The constructions were found using tabu search or as polycirculant graphs, and verified independently in SageMath~\cite{sagemath}; the code for books and wheels is available at \url{https://github.com/gwen-mckinley/ramsey-books-wheels}. The additional lower bounds given in \Cref{lem:bookslemmma} were also established by polycirculant graph constructions, which are available in the same repository. The majority of the upper bounds were found by semidefinite programming using the flag algebra method, as described in \cite{LidP}. The certificates for the upper bounds are available at
\url{http://lidicky.name/pub/gr}. And several upper bounds are from ``bottom-up generation," as described below.\\

Here is more detail on the structure of some lower bound constructions for generalized Ramsey numbers: for \(GR(3,K_4,2)\), one extremal coloring is composed of the Paley graph on 9 vertices in the first color, and the other two colors are 3 triangles each. This is the only coloring with more than 6 symmetries. For \(GR(4,K_4,3)\), the only maximal graph is isomorphic to $3 K_3$ in every color. Joining two arbitrary colors, we find the graph described for \(GR(3,K_4,2)\). One example for \(GR(3,K_5,2)\) is the coloring composed of the (isomorphic) circulant graphs $C(19,[1,7,8])$, $C(19,[2,3,5])$ and $C(19,[4,6,9])$.

\subsection{Bottom-up generation}
For small Ramsey numbers, it is possible to generate all Ramsey graphs by starting from a trivial case (for example the graph on one vertex) and repeatedly adding an extra vertex in every possible way without creating a subgraph isomorphic to one of the avoided graphs.
For two-color Ramsey numbers this can be achieved efficiently with the program \textit{geng} from the \textit{nauty} package written by Brendan McKay~\cite{MCKAY201494}. We wrote a custom plugin for generating Ramsey graphs avoiding wheels and books.
This allowed us to find the values of \(R(W_5,W_7)\) and \(R(W_5,W_9)\) in 90s and 36 CPU-hours respectively. The computation of \(R(B_2,B_8)\) was significantly heavier and required approximately 2 years of CPU time. The counts are presented in \Cref{tableCounts}. To test the latter program, we also generated all \(R(B_2,B_6)\)-graphs (which requires only a trivial modification to the program) and verified that the results were identical to those in~\cite{BlackLevenRadz}.\\

For the generalized Ramsey numbers, multiple colors are needed, which is not directly supported by \textit{nauty}. We therefore used a simple implementation in SageMath~\cite{sagemath} to do the generation. Both  $GR(3,K_4,2)$ and $GR(4,K_4,3)$ could be solved in less than one minute in this way.

\subsection{Flag Algebras}
The new upper bounds (except for \(R(B_2,B_8)\)) were obtained (independently) using the flag algebra method. This method was developed by Razborov~\cite{MR2371204}, and the main idea is to find a certificate for the bound as a sum of squares,
which can be done efficiently using semidefinite programming. 
We used CSDP~\cite{csdp} and MOSEK~\cite{mosek} solvers in this work. 
While the solvers produce numerical results, we created tools to change the solution to an exact arithmetic one in SageMath~\cite{sagemath}. 
The flag algebra method is quite general, and it is naturally applicable to problems where one is interested in asymptotic results.
For example, the Tur\'an density of small hypergraphs  and its variants~\cite{l2norm} is a group of problems where flag algebras led to many new improvements. \\

The Ramsey problems in this paper are not asymptotic, so the method does not apply directly.
Instead, we consider a particular blow-up of the Ramsey graph to transfer
the problem from a small finite setting into an asymptotic one. 
This approach was developed by the first and third author~\cite{LidP}.
In this work we utilize the machinery developed there to obtain more upper bounds.

\subsection{Polycirculant graphs}\label{polycirc}
$k$-polycirculant graphs are graphs of order \(n\) that have a non-trivial automorphism \(\alpha\) such that all vertex orbits under \(\alpha\) have the same size \(n/k\). For a large amount of Ramsey numbers, there exist extremal graphs that are polycirculant. One might also believe that if an extremal graph is unique, it is more likely to have some symmetries.
For example, the unique Ramsey-\((B_2,B_8)\)-graph on 20 vertices has 240 automorphisms and is 4-polycirculant.\\

We used the exhaustive algorithm outlined in~\cite{GOEDGEBEUR2022212} and adapted it slightly to generate all polycirculant Ramsey graphs avoiding books. With this generator, it is for example easy to show that there are exactly seven 2-polycirculant Ramsey-\((B_2,B_9)\)-graphs on 20 vertices and exactly one 3-polycirculant Ramsey-\((B_2,B_{10})\)-graph on 24 vertices (up to isomorphism).

\begin{proof}[Proof of Lemma~\ref{lem:bookslemmma}]
    The upper bounds were proven in \cite{RS78}. The lower bounds are witnessed by 2-polycirculant graphs (as was in fact one of the examples in \cite{RS78}), which are available here: \url{https://github.com/gwen-mckinley/ramsey-books-wheels}
\end{proof}

Note: the polycirculant graph constructions we found all have the extra property that the first circulant block (the subgraph induced by the first vertex-orbit of the 2-polycirculant automorphism) is isomorphic to the complement of the second circulant block.
For most cases there were many 2-polycirculant Ramsey graphs with this property. For example: there are 581 non-isomorphic 2-polycirculant \((B_{12},B_{13})\)-graphs on 50 vertices.
Most of these graphs had no extra automorphisms besides the ones implied by being 2-polycirculant.
This makes it hard to deduce a pattern in these graphs that could generalize to an infinite family.

\subsection{Tabu search}

\subsubsection{Tabu search details}

To prove a Ramsey lower bound $n$, our goal is to construct an edge-coloring of $K_{n-1}$ avoiding forbidden monochromatic subgraphs (or, in the case of the generalized Ramsey numbers, forbidden subgraphs with too few edge colors). For simplicity, we will restrict the discussion here to the traditional 2-color case, but the ideas are exactly the same for the generalized problem.\\

In this setting, tabu search works roughly as follows: we start with a random 2-edge-coloring of $K_{n-1}$ (equivalently, a random graph on $n-1$ vertices), and count the number of monochromatic substructures we are trying to avoid -- this is the ``score'' of the graph. Then at each step, we re-color one edge, choosing whichever would produce the lowest-scoring graph, while rejecting any choice that would produce a graph already visited. This restriction is enforced by use of a ``tabu set,'' which may consist either of forbidden moves (edges recently edited) or of whole graphs/colorings already visited. Here we take the second approach; more details and discussion are given below.\\

There are many variants of tabu search used in practice, but the one implemented here is very simple, and has two main features:
\begin{enumerate}
    \item We do not restrict the length of the tabu -- i.e., we reject any moves that would produce a graph already visited at \textit{any} point during the search. 
    \item We allow the search to run indefinitely, and never restart at a random graph. Instead, we run several instances of the search in parallel, starting from independently sampled random graphs.

\end{enumerate}

This approach has two main advantages: first, because this algorithm has no hyperparameters to tune, it is simple and usable ``out of the box'' for new problems. Second, by allowing the search to run indefinitely, we avoid inadvertently restarting too quickly. Empirically, it seems that the number of steps required to find a successful construction grows quickly with the problem parameters, and it is not immediately obvious how to predict this value. If we ``bid low'' and restart the search too quickly, we may consistently cut it short before it is able to reach a promising area of the search space. On the other hand, we observed no clear empirical disadvantage in allowing the search to continue\footnote{There are a number of open-source hyperparameter tuning tools supporting this kind of experimentation; we used \href{https://optuna.org/}{Optuna} here.}. Indeed, under some conditions, local search algorithms provably achieve optimal speedup via parallelization (as opposed to restarts); this is dependent on the runtime distribution of the particular algorithm. See Section 4.4 and Chapter 6 of \cite{SLS_book} for further discussion and related empirical results.\\

We now return to the details of the choice of tabu set. Most variants of tabu search used in the literature store only a list of recent moves (e.g., edges recently edited), rather than whole graphs/configurations, due in large part to memory constraints \cite{tabu_chapter, SLS_book}. However, if only moves are stored, it is important to limit the length of the tabu set -- otherwise all possible moves will quickly be forbidden. 
On the other hand, storing a full representation of each graph visited (and checking the result of each possible move against this set) is very computationally expensive. To avoid this issue, we instead store the graphs by hashing. This offers a much lower computational cost, while still allowing us to avoid all graphs previously visited. (It should be noted, however, that hash collisions may cause us to avoid some other graphs as well -- we did not observe issues with this in practice, but it may be worth further study.) This general strategy has been seen very occasionally in the literature \cite{hash}; but again, it is more common to store a tabu set of recent moves, in which case the length of the tabu is an important hyperparameter, which must be tuned correctly or dynamically updated for good performance.

\subsubsection{Scoring functions}

Also worth noting: an important ingredient in any local search algorithm is an efficient ``scoring" function; here, this means counting the number of 
``forbidden" substructures in a graph. In fact, rather than re-scoring the entire graph at each step, it is generally more efficient to compute the \textit{change} in score at each step, since most of the graph will be unaffected by editing a single edge.\\ 

For books $B_k$, both operations are quite efficient, and the cost is basically independent of $k$ for $k\geq 2$. Explicitly, choosing a copy of $B_k$ in a graph $G$ reduces to first choosing a ``spine'' $(u,v)\in E(G)$, then choosing any $k$ vertices from the common neighborhood $N(u)\cap N(v)$ to serve as the ``pages''; see Figure~\ref{fig-book}.

\begin{figure}[H]
\begin{center}
    \includegraphics{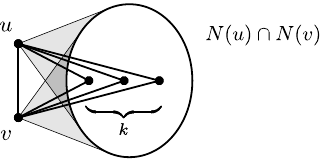}
\end{center}
\caption{Book $B_k$ in a graph}\label{fig-book}
\end{figure}

Notice that for the purpose of counting, we do not need to list all such sets of $k$ vertices; we just need to compute the binomial coefficient $\binom{|N(u)\cap N(v)|}{k}$. Given the value $|N(u)\cap N(v)|$, this operation is very efficient (certainly $O(v(G))$), and in practice, all relevant values of the binomial coefficient can simply be stored in a lookup table. Using similar observations, we can also efficiently compute the ``score change'' for one step in the search (i.e., the change in the number of books from re-coloring an edge), though the details of the implementation are more involved.\\

For wheels $W_k$, we don't exploit any especially ``nice tricks,'' and the cost of our scoring and score change functions are $\Omega(n^k)$ and $\Omega(n^{k-2})$ respectively -- i.e., roughly the cost of enumerating over all subgraphs of the relevant size (where $n$ is the number of vertices in the host graph $G$). Note, however, that significant improvement may be possible; counting wheels reduces largely to counting cycles, where surprisingly efficient algorithms are known \cite{cycle_count}.\\

For generalized Ramsey numbers $GR(r, K_s, t)$, the situation is slightly different; here, we need to count copies of $K_s$ with at most $t$ colors (in an $r$-edge-colored complete graph $G$). Because clique counting is a well-studied problem in non-edge-colored graphs, we use the following strategy: (1) reduce to clique counting in a collection of non-edge-colored graphs, then (2) apply the ``Pivoter" algorithm for fast clique-counting \cite{pivoting} to each graph in the collection. This algorithm also facilitates counting cliques that contain a particular edge, which is useful for our ``score change" computations. (Note: our code implementing this strategy is a rough prototype, so we have not made it publicly available, but it requires relatively minimal adaptations of the code for books and wheels. There is also an open-source implementation of the Pivoter clique-counting algorithm available online \cite{pivoter_code}.)\\

In slightly more detail, we define the scoring function for the generalized Ramsey problem as follows: for each possible set $\mathcal{C}$ of $t$ colors, we construct the (non-edge-colored) graph $G_\mathcal{C}$ consisting of edges whose colors in the original coloring $G$ belong to $\mathcal{C}$. Then we take the ``score" of $G$ to be the sum of the number of cliques $K_s$ in each of these graphs, that is, $\sum_{\mathcal{C}\in\binom{[r]}{t}} (\text{\# copies of $K_s$ in $G_\mathcal{C}$})$. Notice that if a copy of $K_s$ in $G$ uses strictly \textit{fewer} than $t$ colors, it will be counted multiple times by this scoring function, once for each set $\mathcal{C}$ containing its colors. This could be avoided, for example by using inclusion-exclusion, but at a higher computational cost. Also, it is reasonable and perhaps even desirable that this scoring function more heavily penalizes copies of $K_s$ using smaller numbers of colors -- that is, copies that are further from satisfying the desired constraint.

\section{Further directions}

\subsection{Upper bounds}

In the flag algebras applications, we used the most straightforward version of the approach introduced in \cite{LidP}, which generalizes easily. But when calculating an upper bound for a particular Ramsey number, it is also possible to incorporate more constraints. In particular, one could try to include constraints on smaller graphs implied by smaller Ramsey numbers. 
For example when bounding $R(5,5)$, one may use $R(4,5)=25$ when looking at neighborhoods of a fixed vertex.
This approach has been explored by another researcher~\cite{volec}, but it requires considerably more effort. 
In a more distant future it may also be possible to calculate larger instances as computers become faster. 

\subsection{Lower bounds}

\subsubsection{Poycirculant graphs}

It is interesting to ask whether \Cref{lem:bookslemmma} can be extended to larger values of $n$. We proved this lemma by starting a new search for polycirculant graphs for each $n$. There was no indication that the pattern would stop at $n\geq 22$; the generation simply started to take a lot of time. It seems plausible that there is some unified family of graphs giving this bound for all $n$. But again, as noted in \Cref{polycirc}, most of the graphs have no additional symmetry beyond being polycirculant, making it difficult to discern any larger pattern. Note that a partial answer to this question was given in~\cite{wesley2024} where the lemma is generalized to an infinite family consisting of orders that are twice a prime power (that is 1 modulo 4).

\subsubsection{Tabu search}

To extend the tabu search approach used in this work for finding lower bounds, there are a number of natural directions to explore. Perhaps most obvious would be to implement everything in a lower-level language -- all the local search code for wheels and books is currently written in Python. In particular, the score-change functions dominate the total cost, and speeding them up would almost certainly improve performance; relevant here is an extensive literature on subgraph counting (see \cite{subgraph_survey}), and it may be possible to simply incorporate existing open-source solutions.\\

Another small but potentially effective change could be to incorporate weights in the scoring function -- that is, to penalize copies of one forbidden subgraph more than another, depending on their respective sizes and/or how recently they have been edited. This approach has often been used for Ramsey lower bounds, but there appears to be no clear theoretical justification for any particular choice of weights, and some works have even arrived at contradictory conclusions \cite{tankers, exoo13}. This question has been studied more extensively in the SAT literature \cite{clauseWeights,clauseWeightsInitial}, and it is possible some of that work could be fruitfully adapted here.\\

Last but not least, one could try to be more strategic in choosing where to start a local search. One common approach is to build up from smaller graphs: we first find a construction on a smaller number of vertices avoiding ``bad" subgraphs, then add a vertex, and run our local search starting from this new graph (as opposed to starting from a random graph) \cite{whatsR55, exoo4_6}. The hope is that by doing this, we direct the search to a promising area of the (very large) search space. We also tried this method, with limited success, but it is likely worth further exploration. 

\section*{Acknowledgments}
The second author would like to thank David Hemminger and Larry Wilson for many helpful and stimulating discussions, and for their consistent encouragement.
This work used the computing resources at the Center for Computational Mathematics, University of Colorado Denver, including the Alderaan cluster, supported by the National Science Foundation award OAC-2019089.  We also used the Stevin computing infrastructure provided by the VSC (Flemish Supercomputer Center), funded by the Research Foundation Flanders (FWO) and the Flemish Government.

\printbibliography

\newpage
\appendix

\section{Counts and constructions}\label{appendix} 

\begin{table}[H]
\begin{center}
    \begin{tabular}{l|c|cc|cc}
    $n$ & $R(B_2,B_8)$ & $R(W_5,W_7)$ & $R(W_5,W_9)$ & $GR(3,K_4,2)$ & $GR(4,K_4,3)$
    \\ \hline
    1 & 1 & 1 & 1 & 1 & 1 \\
    2 & 2 & 2 & 2 & 1 & 1 \\
    3 & 4 & 4 & 4 & 3 & 3 \\
    4 & 9 & 11 & 11 & 9 & 7 \\
    5 & 22 & 31 & 31 & 34 & 11 \\
    6 & 69 & 130 & 130 & 154 & 12 \\
    7 & 255 & 675 & 723 & 428 & 1 \\
    8 & 1301 & 4868 & 6456 & 556 & 1 \\
    9 & 9297 & 38059 & 87977 & 263 & 1 \\
    10 & 96618 & 251377 & 1627532 & 0 & 0 \\
    11 & 1405777 & 878658 & 27891376 &  &  \\
    12 & 25330324 & 932411 & 250459368 &  &  \\
    13 & 443322144 & 141871 & 509767930 &  &  \\
    14 & 5130980404 & 1161 & 139131233 &  &  \\
    15 & ? & 0 & 46736023 &  &  \\
    16 & ? &  & 18956337 &  &  \\
    17 & ? &  & 272891 &  &  \\
    18 & ? &  & 0 &  &  \\
    19 & 41 &  &  &  &  \\
    20 & 1 &  &  &  &  \\
    21 & 0 &  &  &  &  \\
    \end{tabular}
\end{center}
\caption{Counts for Ramsey graph. The multicolor numbers are given up to color-swapping isomorphism. Because the calculations for $R(B_2,B_8)$ had to be split up into multiple processes, we have no accurate counts for some orders.}\label{tableCounts}
\end{table}

\begin{itemize}
\vbox{
    \item $R(W_5,W_7)\geq 15$.\\
    Graph6-string of a 14-vertex graph with no $W_5$ as a subgraph and no $W_7$ in the complement:

\begin{lstlisting}
Mav?Hwu]`ySZpyyg?
\end{lstlisting}
    
}

\vbox{
    \item $R(W_5,W_9)\geq 18$.\\
    Graph6-string of a 17-vertex graph with no $W_5$ as a subgraph and no $W_9$ in the complement:

\begin{center}
\begin{lstlisting}
PIL{eMI^Jqp[gXkp_|zxOaww
\end{lstlisting}
\end{center}    
}

\vbox{
    \item $R(B_2, B_8) \geq 21.$\\
        Graph6-string of a 20-vertex graph with no $B_2$ as a subgraph and no $B_8$ in the complement:

    \begin{lstlisting}
SXgcISrdSaJQBJs_jp@CWFOV?q}HWOPbc
    \end{lstlisting}
}

\vbox{
    \item $R(B_2, B_9) \geq 22.$\\
        Graph6-string of a 21-vertex graph with no $B_2$ as a subgraph and no $B_9$ in the complement:    
    \begin{lstlisting}
TXhJ?ScLQoHAO]EhcLe_G_nAEXuiBSnW?]?w 
    \end{lstlisting}
}

\vbox{
    \item $R(B_2, B_{10}) \geq 25.$\\
    Graph6-string of a 24-vertex graph with no $B_2$ as a subgraph and no $B_{10}$ in the complement:

\begin{center}
\begin{lstlisting}
W?bFFbw^@{BwgDsAl?lg@U_cl@GlDGUacDih?lTKApSgDqh
\end{lstlisting}
\end{center}
}

\vbox{
    \item $R(B_3, B_6) \geq 19$.\\
    Graph6-string of a 18-vertex graph with no $B_3$ as a subgraph and no $B_{6}$ in the complement:

    \begin{lstlisting}    
QG]cql@_GTeAwhrAVeEaiHv?Sn?
    \end{lstlisting}
}

\vbox{
    \item $R(B_3, B_7) \geq 20$.\\
        Graph6-string of a 19-vertex graph with no $B_3$ as a subgraph and no $B_{7}$ in the complement:
    \begin{lstlisting} 
REf`OcBMI@ozZSyaMGil?ABm_o|jOO
    \end{lstlisting}
}

\vbox{
    \item $R(B_4,B_5) \geq 19$.\\
        Graph6-string of a 18-vertex graph with no $B_4$ as a subgraph and no $B_{5}$ in the complement:
    \begin{lstlisting}
QYMOYLbMIt\GTS_Mtp]_YAtuuoO
    \end{lstlisting}
}

\vbox{
    \item $R(B_4,B_7) \geq 22$.\\
        Graph6-string of a 21-vertex graph with no $B_4$ as a subgraph and no $B_{7}$ in the complement:

     \begin{lstlisting}
TqmoKUaoq\bAK|]oMgORPWJYMCxGye`EmTcY 
    \end{lstlisting}
}

\vbox{
    \item $R(B_5,B_6) \geq 23$.\\
    Graph6-string of a 22-vertex graph with no $B_5$ as a subgraph and no $B_{6}$ in the complement:

     \begin{lstlisting}
U_Ya{gHmOv}QfaSGkhXLXoN]krJqbE^?dvdCkHso     
    \end{lstlisting}
}

\vbox{
    \item $R(B_5,B_7) \geq 25$.\\
    Graph6-string of a 24-vertex graph with no $B_5$ as a subgraph and no $B_{7}$ in the complement:

     \begin{lstlisting}
 W|teicY@kY[EQsKWEJqIIde]`^BZr?zwhGnwaCv`LUHF_UJ
    \end{lstlisting}
}

\vbox{
    \item $R(B_6,B_7) \geq 27$.\\
    Graph6-string of a 26-vertex graph with no $B_6$ as a subgraph and no $B_{7}$ in the complement:
    
 \begin{lstlisting}
 YMQcwl`gEBbKaZK{SbPhN~BNnaA\Q_YyVQ{A]uwEoemTUhkBdkk\T@Y_ 
\end{lstlisting}
}

\vbox{
\item $R(B_6, B_8) \geq 28$.\\
    Graph6-string of a 27-vertex graph with no $B_6$ as a subgraph and no $B_{8}$ in the complement:
    
 \begin{lstlisting}
 ZOp~T_WyXbZ`IjWLOfpZcFgDurCaMkcd]v`gpyHaEFdjjWLiIQtHB]QiZbIG
\end{lstlisting}

}

\vbox{
    \item $R(B_7,B_8) \geq 31$.\\
    Graph6-string of a 30-vertex graph with no $B_7$ as a subgraph and no $B_{8}$ in the complement:
    
\begin{center}
\begin{lstlisting}
]UWsqWecqRCeceqRKcsDjoUn_lJ_lLoUd[Dxj_jMmAkl[FXl[BXUmDkdjbZGl[ZXAtplcDjrZG
\end{lstlisting}
\end{center}
}

\vbox{
    \item $R(B_8,B_8) \geq 33$.\\
    Graph6-string of a 32-vertex graph with no $B_8$ as a subgraph and no $B_{8}$ in the complement:
    
\begin{center}
\begin{lstlisting}
_Uzrpy]RpNQ]q]xN]Rrq]`DnAJ^AJJ`DewPXvAHJ[GsuwPhUwRhJ[JsavB|KUwNhPZa}cavD|GavD|GPZq}c
\end{lstlisting}
\end{center}
}

\vbox{
    \item $GR(3,K_4,2) = 10$.\\
    Adjacency matrix of a 3-edge-colored 9-vertex graph with each $K_4$ containing more than 2 colors.
There are 263 such colorings (up to color permutations).
    \begin{center}
\begin{verbatim}
[[0,2,3,3,3,1,1,1,2]
 [2,0,1,1,2,1,3,3,2]
 [3,1,0,2,3,1,3,2,1]
 [3,1,2,0,1,3,1,2,3]
 [3,2,3,1,0,2,2,3,1]
 [1,1,1,3,2,0,2,3,3]
 [1,3,3,1,2,2,0,1,3]
 [1,3,2,2,3,3,1,0,1]
 [2,2,1,3,1,3,3,1,0]]
    \end{verbatim}\end{center}
}

\vbox{
    \item $GR(4,K_4,2) \geq 15$\\
     Adjacency matrix of a 4-edge-colored 14-vertex graph with each $K_4$ containing more than 2 colors.   
     \begin{center}
\begin{verbatim}
[[0,4,3,1,2,3,4,3,2,2,4,1,4,1]
 [4,0,2,4,1,1,1,3,4,3,2,2,3,1]
 [3,2,0,2,2,3,1,1,4,1,3,3,4,4]
 [1,4,2,0,4,2,2,3,1,3,4,3,1,2]
 [2,1,2,4,0,4,3,4,2,1,1,3,2,1]
 [3,1,3,2,4,0,1,4,1,2,2,4,1,3]
 [4,1,1,2,3,1,0,4,3,2,3,3,4,2]
 [3,3,1,3,4,4,4,0,2,2,1,2,2,1]
 [2,4,4,1,2,1,3,2,0,4,1,1,3,3]
 [2,3,1,3,1,2,2,2,4,0,3,1,1,4]
 [4,2,3,4,1,2,3,1,1,3,0,4,2,3]
 [1,2,3,3,3,4,3,2,1,1,4,0,4,1]
 [4,3,4,1,2,1,4,2,3,1,2,4,0,2]
 [1,1,4,2,1,3,2,1,3,4,3,1,2,0]] 
\end{verbatim}
\end{center}
}

\vbox{
    \item $GR(4,K_4,3) = 10$\\
    Adjacency matrix of a 4-edge-colored 9-vertex graph with each $K_4$ containing more than 3 colors. There is only one such graph on 9 vertices.
\begin{center}
\begin{verbatim}
[[0,2,3,3,1,4,2,1,4]
 [2,0,1,4,3,3,2,4,1]
 [3,1,0,3,4,2,4,2,1]
 [3,4,3,0,2,1,1,4,2]
 [1,3,4,2,0,3,4,1,2]
 [4,3,2,1,3,0,1,2,4]
 [2,2,4,1,4,1,0,3,3]
 [1,4,2,4,1,2,3,0,3]
 [4,1,1,2,2,4,3,3,0]]
\end{verbatim}
\end{center}
}

\vbox{
\item $GR(3,K_5,2) \geq 20$\\
    Adjacency matrix of a 3-edge-colored 19-vertex graph with each $K_5$ containing more than 2 colors.
\begin{verbatim}
[[0,1,2,2,3,2,3,1,1,3,3,1,1,3,2,3,2,2,1],
 [1,0,1,2,2,3,2,3,1,1,3,3,1,1,3,2,3,2,2],
 [2,1,0,1,2,2,3,2,3,1,1,3,3,1,1,3,2,3,2],
 [2,2,1,0,1,2,2,3,2,3,1,1,3,3,1,1,3,2,3],
 [3,2,2,1,0,1,2,2,3,2,3,1,1,3,3,1,1,3,2],
 [2,3,2,2,1,0,1,2,2,3,2,3,1,1,3,3,1,1,3],
 [3,2,3,2,2,1,0,1,2,2,3,2,3,1,1,3,3,1,1],
 [1,3,2,3,2,2,1,0,1,2,2,3,2,3,1,1,3,3,1],
 [1,1,3,2,3,2,2,1,0,1,2,2,3,2,3,1,1,3,3],
 [3,1,1,3,2,3,2,2,1,0,1,2,2,3,2,3,1,1,3],
 [3,3,1,1,3,2,3,2,2,1,0,1,2,2,3,2,3,1,1],
 [1,3,3,1,1,3,2,3,2,2,1,0,1,2,2,3,2,3,1],
 [1,1,3,3,1,1,3,2,3,2,2,1,0,1,2,2,3,2,3],
 [3,1,1,3,3,1,1,3,2,3,2,2,1,0,1,2,2,3,2],
 [2,3,1,1,3,3,1,1,3,2,3,2,2,1,0,1,2,2,3],
 [3,2,3,1,1,3,3,1,1,3,2,3,2,2,1,0,1,2,2],
 [2,3,2,3,1,1,3,3,1,1,3,2,3,2,2,1,0,1,2],
 [2,2,3,2,3,1,1,3,3,1,1,3,2,3,2,2,1,0,1],
 [1,2,2,3,2,3,1,1,3,3,1,1,3,2,3,2,2,1,0]]
 \end{verbatim}
 }

\vbox{
\item $GR(3,K_6,2) \geq 32$\\
    Adjacency matrix of a 3-edge-colored 31-vertex graph with each $K_6$ containing more than 2 colors.
\begin{center}    
\begin{verbatim}
[[0,3,3,3,3,2,1,2,1,1,1,1,3,1,3,2,2,2,3,3,1,3,2,1,2,2,3,3,2,1,1]
 [3,0,2,1,3,2,3,2,2,1,2,2,2,1,1,1,2,2,2,2,1,3,3,3,1,1,1,1,3,3,2]
 [3,2,0,2,1,1,3,3,1,3,2,2,3,3,2,2,1,1,3,1,3,2,1,3,3,1,1,1,1,2,2]
 [3,1,2,0,1,3,1,3,3,1,1,2,3,1,2,3,1,1,1,3,1,2,2,3,2,2,3,3,3,2,1]
 [3,3,1,1,0,1,1,2,1,1,2,1,1,2,3,2,1,2,3,2,2,3,2,2,1,2,1,3,3,3,2]
 [2,2,1,3,1,0,1,2,2,3,1,3,2,2,3,3,1,3,2,2,1,3,2,3,1,1,1,3,1,1,3]
 [1,3,3,1,1,1,0,1,1,2,3,1,3,3,2,2,3,1,1,3,2,2,3,1,1,3,2,2,1,3,1]
 [2,2,3,3,2,2,1,0,2,3,2,3,3,3,1,1,2,2,1,2,3,1,1,3,1,1,3,2,3,2,1]
 [1,2,1,3,1,2,1,2,0,2,2,1,2,2,1,3,1,2,2,3,2,1,2,1,3,2,3,1,1,3,3]
 [1,1,3,1,1,3,2,3,2,0,2,2,3,1,3,1,1,1,1,1,2,1,3,3,3,2,2,2,3,2,1]
 [1,2,2,1,2,1,3,2,2,2,0,2,2,2,2,2,3,2,3,3,1,2,3,1,1,3,1,3,2,1,2]
 [1,2,2,2,1,3,1,3,1,2,2,0,1,3,3,3,1,3,1,2,3,1,1,1,1,2,1,2,3,2,3]
 [3,2,3,3,1,2,3,3,2,3,2,1,0,3,1,2,1,1,2,1,3,1,1,3,1,2,1,1,2,2,2]
 [1,1,3,1,2,2,3,3,2,1,2,3,3,0,1,1,1,3,2,2,3,2,2,2,2,1,1,1,3,1,3]
 [3,1,2,2,3,3,2,1,1,3,2,3,1,1,0,1,2,3,3,1,2,2,1,2,3,3,2,1,1,3,3]
 [2,1,2,3,2,3,2,1,3,1,2,3,2,1,1,0,3,3,3,3,1,1,3,3,1,1,3,2,2,3,2]
 [2,2,1,1,1,1,3,2,1,1,3,1,1,1,2,3,0,2,3,2,2,3,3,2,2,3,1,3,2,3,1]
 [2,2,1,1,2,3,1,2,2,1,2,3,1,3,3,3,2,0,1,1,3,3,2,2,2,1,1,3,1,1,3]
 [3,2,3,1,3,2,1,1,2,1,3,1,2,2,3,3,3,1,0,2,3,3,2,2,2,3,1,3,1,2,1]
 [3,2,1,3,2,2,3,2,3,1,3,2,1,2,1,3,2,1,2,0,2,1,1,2,3,1,3,1,1,2,2]
 [1,1,3,1,2,1,2,3,2,2,1,3,3,3,2,1,2,3,3,2,0,2,1,1,1,1,2,2,3,1,3]
 [3,3,2,2,3,3,2,1,1,1,2,1,1,2,2,1,3,3,3,1,2,0,1,3,1,1,2,3,3,3,1]
 [2,3,1,2,2,2,3,1,2,3,3,1,1,2,1,3,3,2,2,1,1,1,0,2,2,3,1,1,3,3,1]
 [1,3,3,3,2,3,1,3,1,3,1,1,3,2,2,3,2,2,2,2,1,3,2,0,2,2,2,3,3,1,2]
 [2,1,3,2,1,1,1,1,3,3,1,1,1,2,3,1,2,2,2,3,1,1,2,2,0,1,3,3,2,3,3]
 [2,1,1,2,2,1,3,1,2,2,3,2,2,1,3,1,3,1,3,1,1,1,3,2,1,0,1,3,2,2,2]
 [3,1,1,3,1,1,2,3,3,2,1,1,1,1,2,3,1,1,1,3,2,2,1,2,3,1,0,2,3,1,1]
 [3,1,1,3,3,3,2,2,1,2,3,2,1,1,1,2,3,3,3,1,2,3,1,3,3,3,2,0,1,2,1]
 [2,3,1,3,3,1,1,3,1,3,2,3,2,3,1,2,2,1,1,1,3,3,3,3,2,2,3,1,0,3,2]
 [1,3,2,2,3,1,3,2,3,2,1,2,2,1,3,3,3,1,2,2,1,3,3,1,3,2,1,2,3,0,3]
 [1,2,2,1,2,3,1,1,3,1,2,3,2,3,3,2,1,3,1,2,3,1,1,2,3,2,1,1,2,3,0]]
\end{verbatim}
\end{center}
}

\end{itemize}                  
\end{document}